\documentclass[letterpaper,12pt,oneside]{article}

\usepackage{a4wide}
\usepackage[utf8]{inputenc} 
\usepackage[T1]{fontenc} 
\usepackage{textcomp} 
\usepackage[english,french]{babel} 
\usepackage{amsmath}
\usepackage{amssymb}
\usepackage{amsthm}
\usepackage{graphicx}
\usepackage{url}
\usepackage{hyperref}
\usepackage[all]{xy}
\usepackage{nccmath}
\usepackage{color}
\usepackage{mathrsfs,xspace}
\usepackage{comment}
\usepackage{dsfont}
\usepackage{tikz}
\usetikzlibrary{arrows.meta, positioning}
\usetikzlibrary{overlay-beamer-styles}
\usetikzlibrary{fit}
\allowdisplaybreaks

\DeclareMathOperator{\N}{\mathbb{N}}

\DeclareMathOperator{\Homr}{%
\mathscr{H}\kern-0.3em\textit{\footnotesize\kern-0.2em o\kern-0.12em m}}

\theoremstyle{definition}
\newtheorem{definition}{Definition}[section]

\newtheorem{remark}[definition]{Remark}

\theoremstyle{plain}
\newtheorem{proposition}[definition]{Proposition}
\newtheorem{theorem}[definition]{Theorem}
\newtheorem{lemma}[definition]{Lemma}
\newtheorem{corollary}[definition]{Corollary}

\title{On the minimum number of chips allowing an infinite Chip-firing game}
\author{ Christophe Dubussy\footnote{Department of mathematics, University of Namur, christophe.dubussy@unamur.be} \, and Merlin Michalski\footnote{Department of mathematics, University of Namur, merlin.michalski@student.unamur.be}}
\date{\today}

\begin{document}
\selectlanguage{english}
\maketitle

\begin{abstract}
In this article, we provide three formulas allowing to compute the minimum amount of initial chips leading to an infinite Chip-firing game, answering a question originally posed by Björner and Lovász in 1992. These formulas hold for strongly connected directed loop-free multigraphs and generalize what was already known in the Eulerian case. The various proofs heavily rely on a notion of \textit{dynamical bound}, which allows to encode some specific sequences of chip configurations. 
\end{abstract}

\section{Introduction}

The Chip-firing game is a local dynamical game introduced on undirected graphs in \cite{Bjor91} and on directed multigraphs in \cite{Bjor92}, which is highly connected to many subjects of discrete mathematics, such as the abelian sandpile model, Petri nets, abstract rewriting systems and much more. A good review of all these interconnections can be found in \cite{Kliv18}. The underlying dynamical process of the Chip-firing game is defined as follows. Let $G = (V, E)$ be a directed graph with a certain amount of chips initially distributed among its vertices. At each step of the game, one chooses a vertex that has more chips than neighbors and this vertex distributes (\textit{fires}) one chip to each of its neighbors.

\begin{figure}
\centering
\begin{tikzpicture}[
    vertex/.style={
        circle,
        draw,
        inner sep=1.5pt,
        minimum size=18pt
    },
    arrow/.style={
        ->,
        >=Stealth,
        thick
    }
]

\node[vertex] (A) at (1,8) {2};
\node[vertex] (B) at (1,7) {0};
\node[circle,draw,inner sep=1.5pt,minimum size=18pt, color=red] (C) at (0,6) {1};
\node[vertex] (D) at (2,6) {0};

\draw[arrow] (B) -- (A);
\draw[arrow] (C) -- (B);
\draw[arrow] (D) -- (C);
\draw[arrow] (D) -- (B);

\draw[arrow] (0.83,7.7) -- (0.83,7.3);
\draw[arrow] (1.17,7.7) -- (1.17,7.3);
\draw[arrow] (1.1,6.7) -- (1.7, 6.1);
\draw[arrow] (1.3, 6.9) -- (1.9, 6.3);

\node[circle,draw,inner sep=1.5pt,minimum size=18pt, color=red] (E) at (5,8) {2};
\node[vertex] (F) at (5,7) {1};
\node[vertex] (G) at (4,6) {0};
\node[vertex] (H) at (6,6) {0};

\draw[arrow] (F) -- (E);
\draw[arrow] (G) -- (F);
\draw[arrow] (H) -- (G);
\draw[arrow] (H) -- (F);

\draw[arrow] (4.83,7.7) -- (4.83,7.3);
\draw[arrow] (5.17,7.7) -- (5.17,7.3);
\draw[arrow] (5.1,6.7) -- (5.7, 6.1);
\draw[arrow] (5.3, 6.9) -- (5.9, 6.3);

\draw[->, line width=1mm, >=Stealth] (2.5,7) -- (3.5,7);

\node[vertex] (I) at (9,8) {0};
\node[circle,draw,inner sep=1.5pt,minimum size=18pt, color=red] (J) at (9,7) {3};
\node[vertex] (K) at (8,6) {0};
\node[vertex] (L) at (10,6) {0};

\draw[arrow] (J) -- (I);
\draw[arrow] (K) -- (J);
\draw[arrow] (L) -- (K);
\draw[arrow] (L) -- (J);

\draw[arrow] (8.83,7.7) -- (8.83,7.3);
\draw[arrow] (9.17,7.7) -- (9.17,7.3);
\draw[arrow] (9.1,6.7) --  (9.7, 6.1);
\draw[arrow] (9.3, 6.9) -- (9.9, 6.3);

\draw[->, line width=1mm, >=Stealth] (6.5,7) -- (7.5,7);

\node[vertex] (1) at (9,4) {1};
\node[vertex] (2) at (9,3) {0};
\node[vertex] (3) at (8,2) {0};
\node[circle,draw,inner sep=1.5pt,minimum size=18pt, color=red] (4) at (10,2) {2};

\draw[arrow] (2) -- (1);
\draw[arrow] (3) -- (2);
\draw[arrow] (4) -- (3);
\draw[arrow] (4) -- (2);

\draw[arrow] (8.83,3.7) -- (8.83,3.3);
\draw[arrow] (9.17,3.7) -- (9.17,3.3);
\draw[arrow] (9.1, 2.7) -- (9.7, 2.1);
\draw[arrow] (9.3, 2.9) -- (9.9, 2.3);

\draw[->, line width=1mm, >=Stealth] (9,5.5) -- (9,4.5);

\node[vertex] (5) at (5,4) {1};
\node[vertex] (6) at (5,3) {1};
\node[circle,draw,inner sep=1.5pt,minimum size=18pt, color=red] (7) at (4,2) {1};
\node[vertex] (8) at (6,2) {0};

\draw[arrow] (6) -- (5);
\draw[arrow] (7) -- (6);
\draw[arrow] (8) -- (7);
\draw[arrow] (8) -- (6);

\draw[arrow] (4.83,3.7) -- (4.83,3.3);
\draw[arrow] (5.17,3.7) -- (5.17,3.3);
\draw[arrow] (5.1, 2.7) -- (5.7, 2.1);
\draw[arrow] (5.3, 2.9) -- (5.9, 2.3);

\draw[->, line width=1mm, >=Stealth] (7.5,3) -- (6.5,3);

\node[vertex] (9) at (1,4) {1};
\node[vertex] (10) at (1,3) {2};
\node[vertex] (11) at (0,2) {0};
\node[vertex] (12) at (2,2) {0};

\draw[arrow] (10) -- (9);
\draw[arrow] (11) -- (10);
\draw[arrow] (12) -- (11);
\draw[arrow] (12) -- (10);

\draw[arrow] (0.83,3.7) -- (0.83,3.3);
\draw[arrow] (1.17,3.7) -- (1.17,3.3);
\draw[arrow] (1.1, 2.7) -- (1.7, 2.1);
\draw[arrow] (1.3, 2.9) -- (1.9, 2.3);

\draw[->, line width=1mm, >=Stealth] (3.5,3) -- (2.5,3);

\end{tikzpicture}
\caption{An example of finite Chip-firing game}
\label{fig: rule chip-firing game}
\end{figure}
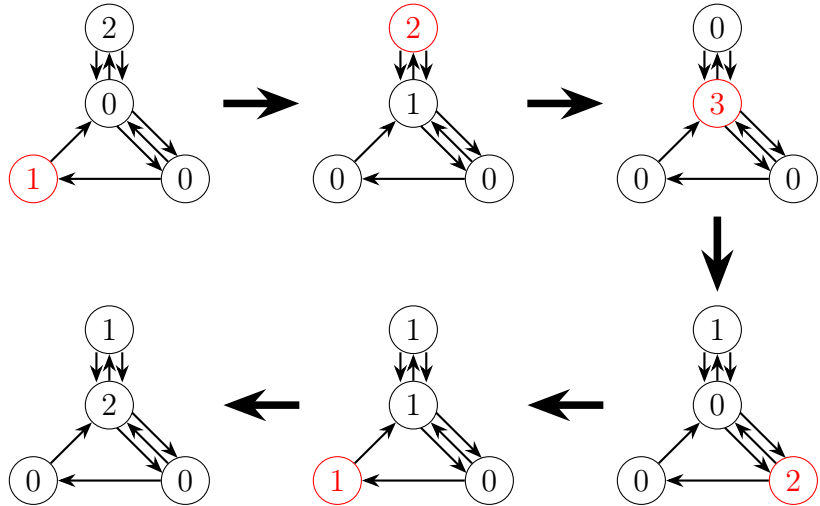

\medskip
Starting from an initial configuration, the game is said to be \textit{terminating}, or \textit{finite}, if no further distributions of chips are possible after a finite number of steps. Otherwise it is said to be \textit{non-terminating} or \textit{infinite}.  This dynamical process immediately raises two questions:

\begin{enumerate}
\item What number of initially distributed chips guarantees that any starting configuration leads
to a finite game?
\item What number of initially distributed chips guarantees that any starting configuration leads
to an infinite game?
\end{enumerate}

In \cite{Bjor91}, the authors definitively answer these questions in the undirected case. If $G$ has $N$ vertices and $M$ undirected edges, then the game is always infinite if strictly more than $2M-N$ chips are distributed, and always finite if strictly fewer than $M$ chips are distributed. Between
these two limits, there is always at least one starting configuration which makes the game finite and another configuration which makes the game infinite. In \cite{Bjor92}, it is noted that one of the two bounds remains trivially identical in the case of directed graphs. Indeed, if $G$ is a directed graph with $N$ vertices and $M$ directed edges, then, by the
pigeonhole principle, if we distribute strictly more than $M-N$ chips at the start, there will always be a vertex capable of firing a chip. The game is therefore infinite. On the other hand, the other question becomes particularly difficult. It can be reformulated as follows:

\medskip
\noindent
\begin{center}
\textit{Given a directed graph, what is the minimum number of chips that must be distributed at the start to guarantee the existence of at least one starting configuration that makes the game infinite?}
\end{center}

Throughout this article, this number will be denoted by c. As we said, in the undirected case we have $c=M.$ In the directed case, it has been proved in \cite{Bjor92} that, if the graph is strongly connected, $c \geq f$ where $f$ is the \textit{feedback number} of the graph, i.e. the minimum number of edges to remove in order to leave $G$ acyclic, and that one has the equality $c = f$ if $G$ is Eulerian.  It is thus immediately clear that it will not be possible to calculate $c$ in a polynomial time, even in the Eulerian case, since determining $f$ is a NP-hard problem \cite{Perr15}. However, the link exhibited with the feedback number paves the way for fast heuristic methods allowing the calculation of $c$ in the Eulerian case, see e.g. \cite{Bran11,Eade93,Gela23}. To the best of our knowledge, the non-Eulerian case has not been treated until now, even though the study of non-terminating configurations has important connections with the discrete Riemann-Roch theory \cite{Hujt19}. Therefore, the main purpose of this paper is to establish several explicit formulas for $c$, in the strongly connected case. 

\medskip
Our key idea is to go backwards, starting from a configuration where all the vertices have zero chips and, at each step, to select a vertex $v$ and determine the minimum number of chips that all the vertices would have had at the previous stage of the game if $v$ had been selected to fire. The sequence of selected vertices is called \textit{a strategy} and the sequence of backwards configurations is called \textit{a dynamical bound} (see Figure \ref{fig: dynamical bound} for an example). To such a bound, one can associate its \textit{gain}, i.e. the limit of the sum of all the chips on the graph as the number of steps tends to infinity. It is shown in Theorem~\ref{First formula} that $c$ is equal to the minimal gain among all the possible strategies. Obviously, such a formula is purely theoretical and not usable in any practical way.

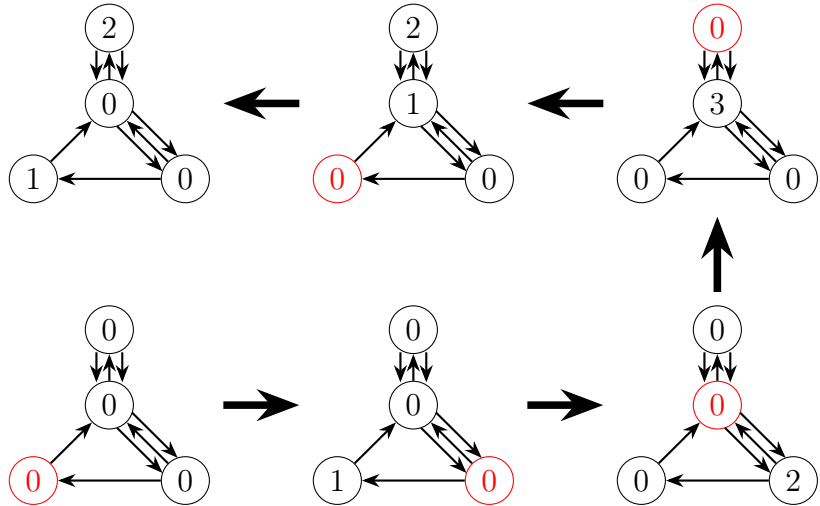
\begin{figure}[h]
\centering
\begin{tikzpicture}[
    vertex/.style={
        circle,
        draw,
        inner sep=1.5pt,
        minimum size=18pt
    },
    arrow/.style={
        ->,
        >=Stealth,
        thick
    }
]

\node[vertex] (A) at (1,8) {2};
\node[vertex] (B) at (1,7) {0};
\node[vertex] (C) at (0,6) {1};
\node[vertex] (D) at (2,6) {0};

\draw[arrow] (B) -- (A);
\draw[arrow] (C) -- (B);
\draw[arrow] (D) -- (C);
\draw[arrow] (D) -- (B);

\draw[arrow] (0.83,7.7) -- (0.83,7.3);
\draw[arrow] (1.17,7.7) -- (1.17,7.3);
\draw[arrow] (1.1,6.7) -- (1.7, 6.1);
\draw[arrow] (1.3, 6.9) -- (1.9, 6.3);

\node[vertex] (E) at (5,8) {2};
\node[vertex] (F) at (5,7) {1};
\node[circle,draw,inner sep=1.5pt,minimum size=18pt, color=red] (G) at (4,6) {0};
\node[vertex] (H) at (6,6) {0};

\draw[arrow] (F) -- (E);
\draw[arrow] (G) -- (F);
\draw[arrow] (H) -- (G);
\draw[arrow] (H) -- (F);

\draw[arrow] (4.83,7.7) -- (4.83,7.3);
\draw[arrow] (5.17,7.7) -- (5.17,7.3);
\draw[arrow] (5.1,6.7) -- (5.7, 6.1);
\draw[arrow] (5.3, 6.9) -- (5.9, 6.3);

\draw[->, line width=1mm, >=Stealth] (3.5,7) -- (2.5,7);

\node[circle,draw,inner sep=1.5pt,minimum size=18pt, color=red] (I) at (9,8) {0};
\node[vertex] (J) at (9,7) {3};
\node[vertex] (K) at (8,6) {0};
\node[vertex] (L) at (10,6) {0};

\draw[arrow] (J) -- (I);
\draw[arrow] (K) -- (J);
\draw[arrow] (L) -- (K);
\draw[arrow] (L) -- (J);

\draw[arrow] (8.83,7.7) -- (8.83,7.3);
\draw[arrow] (9.17,7.7) -- (9.17,7.3);
\draw[arrow] (9.1,6.7) --  (9.7, 6.1);
\draw[arrow] (9.3, 6.9) -- (9.9, 6.3);

\draw[->, line width=1mm, >=Stealth] (7.5,7) -- (6.5,7) ;

\node[vertex] (1) at (9,4) {0};
\node[circle,draw,inner sep=1.5pt,minimum size=18pt, color=red] (2) at (9,3) {0};
\node[vertex] (3) at (8,2) {0};
\node[vertex] (4) at (10,2) {2};

\draw[arrow] (2) -- (1);
\draw[arrow] (3) -- (2);
\draw[arrow] (4) -- (3);
\draw[arrow] (4) -- (2);

\draw[arrow] (8.83,3.7) -- (8.83,3.3);
\draw[arrow] (9.17,3.7) -- (9.17,3.3);
\draw[arrow] (9.1, 2.7) -- (9.7, 2.1);
\draw[arrow] (9.3, 2.9) -- (9.9, 2.3);

\draw[->, line width=1mm, >=Stealth] (9,4.5) -- (9,5.5) ;

\node[vertex] (5) at (5,4) {0};
\node[vertex] (6) at (5,3) {0};
\node[vertex] (7) at (4,2) {1};
\node[circle,draw,inner sep=1.5pt,minimum size=18pt, color=red] (8) at (6,2) {0};

\draw[arrow] (6) -- (5);
\draw[arrow] (7) -- (6);
\draw[arrow] (8) -- (7);
\draw[arrow] (8) -- (6);

\draw[arrow] (4.83,3.7) -- (4.83,3.3);
\draw[arrow] (5.17,3.7) -- (5.17,3.3);
\draw[arrow] (5.1, 2.7) -- (5.7, 2.1);
\draw[arrow] (5.3, 2.9) -- (5.9, 2.3);

\draw[->, line width=1mm, >=Stealth] (6.5,3) -- (7.5,3)  ;

\node[vertex] (9) at (1,4) {0};
\node[vertex] (10) at (1,3) {0};
\node[circle,draw,inner sep=1.5pt,minimum size=18pt, color=red] (11) at (0,2) {0};
\node[vertex] (12) at (2,2) {0};

\draw[arrow] (10) -- (9);
\draw[arrow] (11) -- (10);
\draw[arrow] (12) -- (11);
\draw[arrow] (12) -- (10);

\draw[arrow] (0.83,3.7) -- (0.83,3.3);
\draw[arrow] (1.17,3.7) -- (1.17,3.3);
\draw[arrow] (1.1, 2.7) -- (1.7, 2.1);
\draw[arrow] (1.3, 2.9) -- (1.9, 2.3);

\draw[->, line width=1mm, >=Stealth] (2.5,3) -- (3.5,3);

\end{tikzpicture}
\caption{The six first steps of a dynamical bound}
\label{fig: dynamical bound}
\end{figure}

\medskip
In order to improve this formula, we introduce an algebraic tool which has already been highly used in Chip-firing studies: the \textit{primitive period vector} $v_G$, which is the smallest positive vector cancelling the Laplacian matrix of the graph (see e.g. \cite{Bjor92,Toth18,Toth22}). Each component of $v_G$ gives the minimum number of time each vertex must be activated in order to go back to the initial configuration. Hence, the \textit{period length} $P = ||v_G||_1$ gives the total number of activations in a cycle of an infinite game.  Thanks to this, we are able to restrict the set of strategies to some particular $P$-periodic strategies. Moreover, we show that the minimal gain is obtained after $P-1$ steps (see Theorem~\ref{Second formula}). Although unsurprising, given the interpretation of $P$, this result requires a particularly long proof. 

\medskip
Finally, we show that $c$ can also be computed thanks to a more intrinsic property of the graph, which we shall call \textit{primitive feedback number}, as it relies on $v_G$ but is equal to the traditional feedback number if $G$ is Eulerian. This characterisation uses all the previous results about the dynamical bounds and involves the construction of an auxiliary graph, called the \textit{primitive extension} of $G$, which is built by duplicating all the vertices according to the components of $v_G$. (See Definition~\ref{prim} and Theorem~\ref{Genfeednum}.) In particular, this result shows that the equality $c=f$ in the Eulerian case is actually a manifestation of a non-Eulerian phenomenon, governed by the primitive period vector.

\section{Background and notations}

Throughout this article, we use the convention $\N = \{0,1,2,3,\dots\}$ and $\N_+ = \N \backslash \{0\}.$

\medskip
\noindent
If $V$ is a finite set, we note $|V|$ its cardinal.

\medskip
\noindent
For any $j,k \in \N$, the set of all natural numbers greater than $j$ and less than $k$ is denoted by $[[j,k]].$

\medskip
\noindent
A directed multigraph is an ordered pair $G=(V,E)$ where $V$ is a finite set of vertices and $E : V \times V \to \N$ a function counting the number of directed edges between ordered pairs of vertices. For all $v \in V$, we set as usual $$d^+(v) = \sum_{v' \in V} E(v,v') \quad \text{and} \quad d^-(v) = \sum_{v' \in V} E(v',v),$$ the \textit{outdegree} and \textit{indegree} of $v$.  

\medskip
\noindent
A directed multigraph is said to be \textit{Eulerian} if $d^+(v)=d^-(v)$ for all $v \in V.$

\medskip
\noindent
A directed multigraph is said to be \textit{loop-free} if $E(v,v)=0$ for all $v \in V.$

\medskip
\noindent
A directed multigraph is said to be \textit{strongly connected} if, for any ordered pair $(v,w) \in V \times V$ where $v \neq w$, there exist a natural number $p \in \N_+$ and a finite sequence $v_0,\dots,v_p$ of vertices such that $v_0=v, v_p=w$ and $E(v_i,v_{i+1}) \geq 1$ for all $i \in [[0,p-1]].$ 

\medskip
We shall now fix some notations with respect to the Chip-firing game. Let $G =(V,E)$ be a directed multigraph. A configuration (of chips) on $G$ is any function from $V$ to $\N$. Given a configuration, a vertex $v$ of $G$ can \textit{fire} if it has more chips than $d^+(v).$ Given an initial configuration of chips $C_0$, a \textit{legal game} is a function $C : \N \to \N^V$ such that 

\begin{enumerate}
\item[$\bullet$] $C(0)= C_0.$
\item[$\bullet$] $C(n+1)= C(n)$ if no vertices can fire a chip given the configuration $C(n)$.
\item[$\bullet$] If a vertex can fire given the configuration $C(n)$ then there is such a vertex $v^*$ such that $C(n+1)(v^*) = C(n)(v^*)-d^+(v^*)$ and $C(n+1)(v) = C(n)(v) + E(v^*,v)$ for all $v \neq v^*.$ One says that $v^*$ is fired or activated.
\end{enumerate}

\noindent
A legal game is said to be \textit{finite} or \textit{terminating} if the sequence $(C(n))_{n \in \N}$ stabilizes at some point. Otherwise it is said to be \textit{infinite} or \textit{non-terminating}. Obviously, at any point of a legal game, the total amount of chips remains unchanged, that is to say $$\sum_{v\in V} C(n)(v) = \sum_{v \in V} C_0(v), \quad \forall n \in \N.$$

The following theorem is of primary importance.
\begin{theorem}[\cite{Bjor92}, Theorem 1.1.]
Given a directed multigraph G and an initial configuration of chips $C_0$, either every legal game can be continued indefinitely, or every legal game terminates after the same number of moves with the same final position. The number of times a given vertex is activated is the same in every finite legal game.
\end{theorem}

As a consequence, given $C_0$, if a legal game is infinite there are no finite legal games. Therefore, we define the number $c$ as the smallest natural number such that there is an initial configuration $C_0$ and an infinite legal game $C$ starting at $C_0$ with $c = \sum_{v \in V} C_0(v).$ 

\medskip

Finally, we introduce some algebraic tools that are particularly useful when studying Chip-firing games. If $G=(V,E)$ is a directed multigraph and if $v_1,\dots,v_N$ is an ordering of $V$, we define the Laplacian matrix $\mathcal{L}$ by \[
    \mathcal{L}_{i,j} = \begin{cases}
        d^+(v_i)-E(v_i,v_i) & \textup{if } i = j,\\
        -E(v_j,v_i) & \textup{if } i \ne j.
    \end{cases}
    \]  
Since the loops do not contribute to the Laplacian matrix, we will only work with loop-free graphs in order to lighten certain proofs. In the case where $G$ is strongly connected, Proposition $4.1.$ in \cite{Bjor92} entails that the null-space $$\mathcal{K} = \{ u \in \mathbb{N}_+^N  : \mathcal{L}u = 0 \}$$ contains a unique vector $v_G$, called \textit{primitive period vector}, such that $v_G(i) \leq u(i)$ for all $i \in [[1,N]]$ and all $u \in \mathcal{K}.$ This is actually a consequence of the Perron-Frobenius theorem applied to $DI-\mathcal{L}$ where $D$ is the maximum outdegree of $G$ (see e.g. \cite{Minc88}). Moreover, all the components of $v_G$ are equal to $1$ if $G$ is Eulerian. We set $$P = ||v_G||_1 = \sum_{i=1}^N v_G(i)$$ and call this number \textit{the period length} of the graph.

\section{Strategies and dynamical bounds: a first theoretical formula}

In this section, we work with a directed loop-free multigraph $G = (V,E)$.

\begin{definition}
A \textit{strategy} is any function $g : \N \to V.$
\end{definition}

\begin{definition}
Let $g$ be a strategy. \textit{The dynamical bound associated to} $g$ is the function $B_g : \mathbb{N} \rightarrow {\mathbb{N}}^V$ defined recursively as follows:
    \begin{itemize}
        \item[$\bullet$] $B_g(0)(v) = 0$ for all $v \in V$.
        \item[$\bullet$] $B_g(n+1)(v) = B_g(n)(v) + d^+(v)$ if $v = g(n)$.
        \item[$\bullet$] $B_g(n+1)(v) = \max(B_g(n)(v) - E(g(n),v), 0)$ if $v \ne g(n)$.
    \end{itemize}
\end{definition}

\begin{definition}
Let $g$ be a strategy and $B_g$ its associated dynamical bound. We define a function $\delta_g : V \times \mathbb{N} \rightarrow \mathbb{N}$ by $$\delta_g(v,n) = \max(0,E(g(n),v) - B_g(n)(v)).$$
\end{definition}

\begin{proposition}\label{delta}
Let $g$ be a strategy and $B_g$ its associated dynamical bound. For all $n \in \mathbb{N}$,
    \[
    \sum_{v \in V} \left(B_g(n+1)(v) - B_g(n)(v) \right) = \sum_{v \in V} \delta_g(v,n).
    \]
\end{proposition}

\begin{proof}
Let $n \in \N$ and let us set $v^* = g(n)$. Then, by definition of $B_g$,
    \begin{align*}
        \sum_{v \in V} (B_g(n+1)(v) - B_g(n)(v)) &= \sum_{v \in V\setminus \{v^*\}} \left( \max( B_g(n)(v) - E(g(n),v),0) - B_g(n)(v) \right) + d^+(v^*)\\
        &= \sum_{v \in V\setminus \{v^*\}} \left(\max(0,E(g(n),v) - B_g(n)(v))  - E(g(n),v) \right) + d^+(v^*)\\
        &=  \sum_{v \in V\setminus \{v^*\}} (\delta_g(v,n) - E(g(n),v)) + d^+(v^*)\\
        &= \sum_{v \in V\setminus \{v^*\}} \delta_g(v,n) -\sum_{v \in V\setminus \{v^*\}}  E(v^*,v) + d^+(v^*)\\
        &=  \sum_{v \in V\setminus \{v^*\}} \delta_g(v,n)\\
        &= \sum_{v \in V} \delta_g(v,n),
    \end{align*}
the last equality coming from $$\delta_g(v^*,n) = \max(0,E(v^*,v^*) - B_g(n)(v^*)) = \max(0, - B_g(n)(v^*)) = 0.$$
\end{proof}

\begin{corollary}\label{globalgrowth}
Let $g$ be a strategy and $B_g$ its associated dynamical bound. For all $n \in \mathbb{N}$, 
    \[
    \sum_{v \in V} B_g(n+1)(v) \ge \sum_{v \in V} B_g(n)(v).
    \]
    Moreover, if
    \[
    \sum_{v \in V} B_g(n+1)(v) = \sum_{v \in V} B_g(n)(v),
    \]
    then $B_g(n)(v) \ge E(g(n),v)$ for all $v \in V.$
\end{corollary}

Since the sequence $\left(\sum_{v \in V} B_g(n)(v)\right)_{n \in \N}$ is increasing, it has a limit in $\N \cup \{+\infty\}.$

\begin{definition}
Let $g$ be a strategy. \textit{The gain} associated to $g$ is the element $h_g \in \N \cup \{+\infty\}$ defined by $$h_g = \lim_{n\to +\infty} \sum_{v \in V} B_g(n)(v).$$
\end{definition}

\begin{definition}
We set $h^* = \min \{h_g : g \, \text{is a strategy}\}.$ A strategy $g$ is said to be \textit{optimal} if $h_g = h^*.$
\end{definition}

The main objective of this section is to show that $c$ is actually equal to $h^*$. The proof is in two parts.

\begin{proposition}
$h^* \leq c.$
\end{proposition}
\begin{proof}
By definition of $c$, there is an initial configuration $C_0$ and a function $C : \N \to \N^V$ such that
    \[
    \sum_{v \in V} C_0(v) = c
    \] 
and $C$ is an infinite legal game. Since $C$ can only take a finite number of different values, there exist $n \in \mathbb{N}$ and $r \in \N_+$ such that $C(n) = C(n+r)$. Let $g_1 : \mathbb{N} \rightarrow V$ be the sequence of vertices activated by $C$ and let $g_2 : \N \to V$ be the strategy defined by
    \[
    g_2(k) = g_1(n+r -1 - \overline{k}),
    \]
where $\overline{k}$ is the remainder of the euclidean division of $k$ by $r$. We shall now show that $h_{g_2} \leq c$, which is enough, since $h^* \leq h_{g_2}.$ In order to do that, we first need to prove that for all $k \in \mathbb{N}$ and $v \in V$, one has
    \[
    B_{g_2}(k)(v) \le C(n+r-\overline{k})(v). \quad (*)
    \]
    This is done by induction on $k$. The case $k = 0$ is clear since $B_{g_2}(0)(v) = 0$. Let us now assume that the result holds for $k \in \mathbb{N}$ and set
    \[
    v^* = g_2(k) = g_1(n+r -1 - \overline{k}).
    \]
    By definition of a legal game, we then have
    \begin{align*}
        B_{g_2}(k+1)(v^*) &= B_{g_2}(k)(v^*) + d^+(v^*) \le C(n+r-\overline{k})(v^*) + d^+(v^*)\\
        &= C(n+r-\overline{k} - 1)(v^*) = C(n+r - \overline{k+1})(v^*),
    \end{align*}
where we used the equality $C(n) = C(n+r)$ in the case where $\overline{k} = r-1$. This proves the induction step for $v^*$. If $v \ne v^*$, then, in a similar way
    \begin{align*}
        B_{g_2}(k+1)(v) &= \max(B_{g_2}(k)(v) - E(v^*,v), 0) \le \max(C(n+r-\overline{k})(v) - E(v^*,v),0)\\
        &= \max(C(n+r-\overline{k} - 1)(v),0) = C(n+r-\overline{k} - 1)(v)\\
         &= C(n+r - \overline{k+1}).
    \end{align*}
Now, by summing $(*)$ over all $v \in V$, one gets 
    \[
    \sum_{v \in V} B_{g_2}(k)(v) \le \sum_{v \in V} C(n+r-\overline{k})(v) = \sum_{v \in V} C_0(v)= c.
    \]
Therefore
    \[
    h_{g_2} = \lim_{k \rightarrow \infty}  \sum_{v \in V} B_{g_2}(k)(v) \le c
    \]
and the conclusion follows.
\end{proof}

\begin{proposition}
$c \leq h^*$
\end{proposition}
\begin{proof}
By definition of $h^*$, there is an optimal strategy $g$ such that $h_g = h^*.$ Moreover, by the previous proposition, we know that $h^*$ is finite. Hence, since $h^*$ is the limit of an increasing sequence of natural numbers, this sequence must stabilize, i.e. there exists $n^* \in \mathbb{N}$ such that
    \[
    \sum_{v \in V} B_{g}(n)(v) = h^*,
    \]
    for all $n \geq n^*.$ This implies that there exist $n \ge n^*$ and $r \in \N_+$ such that $B_{g}(n) = B_{g}(n+r)$. We shall now define a sequence of configurations $C :  \mathbb{N} \rightarrow \mathbb{N}^V$ by 
    \[
    C(k) = B_{g}(n+r  - \overline{k}).
    \]
and show that $C$ is an infinite legal game with initial configuration equal to $B_{g}(n)$. Firstly, it is clear that $C(0) =  B_{g}(n+r) = B_{g}(n)$. Secondly, for all $k \in \N$, we need to find a vertex $v^*$ that can be activated in order to switch from $C(k)$ to $C(k+1).$ We claim that $v^* = g(n+r-\overline{k}-1)$ is such a vertex. Indeed
    \begin{align*}
        C(k)(v^*) &= B_{g}(n+r  - \overline{k})(g(n+r-\overline{k}-1))\\
        &= B_{g}(n+r  - \overline{k} - 1)(g(n+r-\overline{k}-1)) + d^+(g(n+r-\overline{k}-1))\\
        &\ge d^+(v^*),
    \end{align*} which means that it can be activated. Moreover,
    \begin{align*}
        C(k+1)(v^*) &= B_{g}(n+r  - \overline{k} - 1)(g(n+r-\overline{k}-1))\\
        &= B_{g}(n+r  - \overline{k} )(g(n+r-\overline{k}-1)) - d^+(g(n+r-\overline{k}-1))\\
        &= C(k)(v^*) - d^+(v^*),
    \end{align*}
    where we used the equality $B_{g}(n) = B_{g}(n+r)$ in the case where $\overline{k} = r-1$. Finally, for all $v \ne v^*$, since 
    \[
    \sum_{v \in V} B_{g}(n+r  - \overline{k} - 1)(v) = h^* = \sum_{v \in V} B_{g}(n+r  - \overline{k} )(v),
    \]
    we can use Corollary~\ref{globalgrowth} to deduce that $B_{g}(n+r  - \overline{k} - 1)(v) \ge E(v^*,v)$ and thus that 
    \[
    B_{g}(n+r  - \overline{k} - 1)(v) = B_{g}(n+r  - \overline{k} )(v) + E(v^*,v).
    \]
    Therefore, for all $v \ne v^*$,
    \begin{align*}
        C(k+1)(v) &= B_{g}(n+r  - \overline{k} - 1)(v) = B_{g}(n+r  - \overline{k} )(v) + E(v^*,v) = C(k)(v) + E(v^*,v),
    \end{align*}
which proves that $C$ is an infinite legal game with initial configuration equal to $B_{g}(n)$. Hence
    \[
    c \leq \sum_{v \in V} B_{g}(n)(v) = h^*
    \]
and the conclusion follows.
\end{proof}

By combining the last two propositions, we obtain

\begin{theorem}\label{First formula}
Let $G = (V,E)$ be a directed loop-free multigraph. The number $c$ is equal to $h^*$, i.e. to $$\min_{g \in \mathcal{G}} \left(\lim_{n\to +\infty} \sum_{v \in V} B_g(n)(v) \right),$$ where $\mathcal{G}$ denotes the set of all strategies.
\end{theorem}

Obviously, this theorem cannot really be used to compute $c$ in any practical way. That's why the next section is devoted to a refinement of this formula which allows to drop the limit sign and to strongly restrict the set $\mathcal{G}.$

\section{Reformulation using the primitive period vector}

From now and until the end of this article, we work with a strongly connected directed loop-free multigraph $G = (V,E)$. We set $N = |V|$ and choose an ordering $v_1,\dots,v_N$ of $V$. We denote as usual by $v_G$ the primitive period vector of $G$ and by $P$ its period.

\begin{definition}
The elements of the set
    \[
    \mathcal{F} = \{\sigma \in V^{[[0,P-1]]} : |\sigma^{-1}(\{v_i\})| = v_G(i) \quad \forall i \in[[1,N]] \}.
    \]
are called \textit{primitive functions}. A strategy $g : \N \to V$ is said to be \textit{good} if there is some $\sigma \in \mathcal{F}$ such that $g(k) = \sigma(\overline{k})$ for all $k \in \N,$ where $\overline{k}$ is the remainder of the euclidean division of $k$ by $P.$ If $\sigma$ is a primitive function, we denote by $g_{\sigma}$ the associated good strategy, by $B_{\sigma}$ the associated dynamical bound and by $h_{\sigma}$ the associated gain.
\end{definition}

Theorem~\ref{First formula} can be improved as follows:

\begin{theorem}\label{Second formula}
One has $$c = \min_{\sigma \in \mathcal{F}} \left(\sum_{v \in V} B_{\sigma}(P-1)(v) \right).$$
\end{theorem}

The proof is particularly long and technical. For the reader's convenience, we cut it in a dozen of lemmas. 

\begin{lemma}\label{EvolDynBound}
Let $g$ be a strategy, $r \in \N_+$, $n \in \N$ and $v \in V$. Let us set
    \[
    K = \{j \in [[0,r-1]] : g(n+j) = v\} \quad L = [[0,r-1]] \setminus K
    \]
    and $k = |K|$. We then have
    \[
    B_g(n+r)(v) = B_g(n)(v) + k *d^+(v) - \sum_{j \in L} \min(B_g(n+j)(v), E(g(n+j), v)).
    \]
\end{lemma}
\begin{proof}
By induction on $r \in \N_+.$ If $r=1$, we have two cases.

\medskip
\noindent
\textbf{Case 1: $g(n) = v$}. In that case, $K = \{0\}$, $k = 1$ and $L = \emptyset$. Hence, by definition,
    \[
    B_g(n+1)(v) = B_g(n)(v) + d^+(v) = B_g(n)(v) + k *d^+(v) - \sum_{j \in L} \min(B_g(n+j)(v), E(g(n+j), v)).
    \]\textbf{Case 2: $g(n) \ne v$}. In that case, $K = \emptyset$, $k = 0$ et $L = \{0\}$. Hence, by definition,
    \begin{align*}
        B_g(n+1)(v) &= \max(B_g(n)(v) - E(g(n),v), 0)\\
        &= B_g(n)(v) - \min(B_g(n)(v), E(g(n), v))\\
        &= B_g(n)(v) + k *d^+(v) - \sum_{j \in L} \min(B_g(n+j)(v), E(g(n+j), v)).
    \end{align*}
Let us now prove the induction step. Assume the equality holds for $r \in \N_+$ and set \[
    K = \{j \in [[0,r]] : g(n+j) = v\} \quad L = [[0,r]] \setminus K
    \]
    with $k = |K|$. We also set
    \[
    K' = \{j \in [[0,r-1]] : g(n+j) = v\} \quad L' = [[0,r-1]] \setminus K'
    \] with $k' = |K'|$ and
    \[
    K^* = \begin{cases}
        \emptyset &\textup{if } g(n+r) \ne v\\
        \{0\} &\textup{if } g(n+r) = v
    \end{cases} \quad L^* = \begin{cases}
        \{0\} &\textup{if } g(n+r) \ne v\\
        \emptyset &\textup{if } g(n+r) = v
    \end{cases}
    \]
    with $k^* = |K^*|$. By construction, one has
    \[
    K = \begin{cases}
        K' \cup \{r\} &\textup{if } K^* = \{0\}\\
        K'  &\textup{if } K^* = \emptyset\\
    \end{cases} \quad L = \begin{cases}
        L' \cup \{r\} &\textup{if } L^* = \{0\}\\
        L'  &\textup{if } L^* = \emptyset\\
    \end{cases}
    \] and $k=k'+k^*.$ 
If we use the equality obtained in the initial step with $n$ replaced by $n+r$, we get
    \[
    B_g(n + r +1)(v) = B_g(n+r)(v) + k^* *d^+(v) - \sum_{j \in L^*} \min(B_g(n+r+j)(v), E(g(n+r+j), v)).
    \]
    Moreover, by the induction hypothesis we have
    \[
    B_g(n+r)(v) = B_g(n)(v) + k' *d^+(v) - \sum_{j \in L'} \min(B_g(n+j)(v), E(g(n+j), v)).
    \]
Using these two equalities, we finally obtain
    \begin{align*}
         B_g(n + r +1)(v) &= B_g(n+r)(v) + k^* *d^+(v) - \sum_{j \in L^*} \min(B_g(n+r+j)(v), E(g(n+r+j), v))\\
         &= B_g(n)(v) + k' *d^+(v) - \sum_{j \in L'} \min(B_g(n+j)(v), E(g(n+j), v)) \\
         &\quad+ k^* *d^+(v) - \sum_{j \in L^*} \min(B_g(n+r+j)(v), E(g(n+r+j), v))\\
         &= B_g(n)(v) + k *d^+(v) - \sum_{j \in L} \min(B_g(n+j)(v), E(g(n+j), v)),
    \end{align*}
hence the conclusion.
\end{proof}

The two following lemmas are capital and use the defining properties of good strategies to show that the latter attain their gain very quickly. 

\begin{lemma}\label{suffgood}
Let $\sigma \in \mathcal{F}$ be a primitive function, $v_{\iota} \in V$ with $\iota \in [[1,N]] $ and $n \in \mathbb{N}$. Let us also set
    \[
    Z = \{j \in [[0,n-1]] : \sigma(\overline{j}) = v_{\iota} \}.
    \]
If $|Z| \ge v_G(\iota)$, then $B_{\sigma}(n)(v_{\iota}) \ge E(g_{\sigma}(n), v_{\iota} )$.
\end{lemma}
\begin{proof}
Let $v_{i^*} = g_{\sigma}(n)$. If $v_{i^*} = v_{\iota}$, then $E(g_{\sigma}(n), v_{\iota} ) = 0$ and the result is trivially proved. We can thus assume that $v_{i^*} \ne v_{\iota}$ (which means $i^* \ne \iota $). For all $l \in [[0,n-1]]$, let us set
    \[
    Z(l) = \{j \in [[l,n-1]] : \sigma(\overline{j}) = v_{\iota} \}.
    \]
    By the hypothesis, we know that there exists some $l \in [[0,n-1]] $ such that $|Z(l)| =  v_G(\iota)$. We can therefore define
    \[
    m = \max \{l \in [[0,n-1]] : |Z(l)| = v_G({\iota}) \}.
    \]
    Now, let us set $r= n - m$. We shall prove that $r \le P -1$. If $n \le P - 1$, this is trivial since $r \leq n$. If $P = 1$, then $N = 1$ (since $P \ge N$), which contradicts $v_{i^*} \ne v_{\iota}$. Otherwise, we have
    \begin{align*}
        |Z(n+1-P)| &= |\{j \in [[n+1-P,n-1]] : \sigma(\overline{j}) = v_{\iota} \}| = |\{j \in [[n+1,n-1+P]] : \sigma(\overline{j}) = v_{\iota} \}|\\
        &= |\{j \in [[n,n-1+P]] : \sigma(\overline{j}) = v_{\iota} \}|, \textup{ since } g_{\sigma}(n) = \sigma(\overline{n}) \ne v_{\iota} \\ 
        &= |\{j \in [[0,P-1]] : \sigma(\overline{j}) = v_{\iota} \}|, \textup{ since } [[n,n-1+P]] = [[0,P-1]] \mod{P}\\
        &= |\{j \in [[0,P-1]] : \sigma(j) = v_{\iota} \}| = |\sigma^{-1}(\{v_{\iota}\})| = v_G(\iota).
    \end{align*}
By the definition of $m$, this means that $m \ge n+1-P$ and thus that $r \le P-1$. For all $i \in [[1,N]],$ we set
    \[
    K_i = \{ j \in [[0,r-1]] : g_{\sigma}(m+j) = v_i\} \quad \text{and} \quad L = \bigcup_{i \ne \iota} K_i = [[0,r-1]] \backslash K_{\iota}.
    \]
    By Lemma~\ref{EvolDynBound}, we obtain
    \begin{align*}
        B_{\sigma}(n)(v_{\iota}) &= B_{\sigma}(m)(v_{\iota}) + |K_{\iota}| *d^+(v_{\iota}) - \sum_{j \in L}\min(B_{\sigma}(m+j)(v_{\iota}), E(g_{\sigma}(m+j), v_{\iota}))\\
        &\ge |K_{\iota}| *d^+(v_{\iota}) - \sum_{j \in L} E(g_{\sigma}(m+j), v_{\iota}) = |K_{\iota}| *d^+(v_{\iota}) - \sum_{\underset{i \ne \iota}{i =1}}^N \sum_{j \in K_i} E(g_{\sigma}(m+j), v_{\iota})\\
        &= |K_{\iota}| *d^+(v_{\iota}) - \sum_{\underset{i \ne \iota}{i =1}}^N \sum_{j \in K_i} E(v_i, v_{\iota}) = |K_{\iota}| *d^+(v_{\iota}) - \sum_{i =1}^N \sum_{j \in K_i} E(v_i, v_{\iota})\\
        &=  |K_{\iota}| *d^+(v_{\iota}) - \sum_{i =1}^N |K_i|* E(v_i, v_{\iota}).\\
    \end{align*}
Our final goal is to bound $|K_i|$ for all $i \in [[1,N]] $. If $i = \iota$, we have
    \begin{align*}
        |K_{\iota}| &= |\{j \in [[0,r-1]] : g_{\sigma}(m+j) = v_{\iota}\}| = |\{j \in [[0,r-1]] : \sigma(\overline{m+j}) = v_{\iota}\}| \\
        &= |\{j \in [[m,n-1]] : \sigma(\overline{j}) = v_{\iota}\}| = |Z(m)| = v_G(\iota).
    \end{align*}
    If $i = i^*$, we can use $g_{\sigma}(n) = v_{i^*}$ and $r \le P-1$ to deduce that
    \begin{align*}
        |K_{i^*}| &= |\{j \in [[0,r-1]] : g_{\sigma}(m+j) = v_{i^*}\}| \le |\{j \in [[0,P-1]] : g_{\sigma}(m+j) = v_{i^*}\}| - 1\\
        &= |\{j \in [[0,P-1]] : \sigma(\overline{m+j}) = v_{i^*}\}| - 1 = |\{j \in [[0,P-1]] : \sigma(\overline{j}) = v_{i^*}\}| -1\\
        &= |\{j \in [[0,P-1]] : \sigma(j) = v_{i^*}\}| -1 = |\sigma^{-1}(\{v_{i^*}\})| -1 = v_G(i^*) - 1.
    \end{align*}
    Finally, in a similar fashion, if $i \ne \iota$ et $i \ne i^*$ we have $|K_{i}| \leq v_G(i).$ We can now write
    \begin{align*}
        B_{\sigma}(n)(v_{\iota}) &\ge |K_{\iota}| *d^+(v_{\iota}) - \sum_{i =1}^N |K_i|* E(v_i, v_{\iota})\\
        &\ge v_G(\iota) * d^+(v_{\iota}) - \left( \sum_{i =1}^N v_G(i)*E(v_i, v_{\iota}) \right) + E(v_{i^*}, v_{\iota})\\
        &= (\mathcal{L}v_G)(\iota) + E(v_{i^*}, v_{\iota})\\
        &= E(v_{i^*}, v_{\iota}),
    \end{align*}
where, at the last step, we used the fact that $v_G$ cancels the Laplacian matrix. Hence the conclusion.
\end{proof}

\begin{lemma}\label{gaingood}
For any primitive function $\sigma \in \mathcal{F}$ we have
    \[
    h_{\sigma} = \sum_{v \in V} B_{\sigma}(P-1)(v).
    \]
\end{lemma}
\begin{proof}
We have to show that
    \[
    \sum_{v \in V} B_{\sigma}(n+1)(v) = \sum_{v \in V} B_{\sigma}(n)(v)
    \] for all $n \geq P-1.$ Let $n \geq P-1$. For all $v_i \ne \sigma(\overline{n})$, we again set $$Z = \{j \in [[0,n-1]] : \sigma(\overline{j}) = v_i\}$$ and we have
    \begin{align*}
        |Z| &= |\{j \in [[0,n-1]] : \sigma(\overline{j}) = v_i\}|= |\{j \in [[0,n]] : \sigma(\overline{j}) = v_i\}|\\
        &\ge |\{j \in [[0,P-1]] : \sigma(\overline{j}) = v_i\}| = |\sigma^{-1}(\{v_i\})| = v_G(i).
    \end{align*}
Hence, by Lemma~\ref{suffgood}, we deduce that $B_{\sigma}(n)(v) \ge E(g_{\sigma}(n),v)$ for all $v \ne \sigma(\overline{n})$. Since $\sigma(\overline{n}) = g_{\sigma}(n)$, we then have
    \begin{align*}
        \sum_{v \in V} B_{\sigma}(n+1)(v) &= \sum_{v \in V\setminus \{g_{\sigma}(n)\}} B_{\sigma}(n+1)(v) + B_{\sigma}(n+1)(g_{\sigma}(n))\\
       	&= \sum_{v \in V\setminus \{g_{\sigma}(n)\}} \left(B_{\sigma}(n)(v) - E(g_{\sigma}(n),v)\right) + B_{\sigma}(n)(g_{\sigma}(n)) + d^+(g_{\sigma}(n))\\
        &= \sum_{v \in V} B_{\sigma}(n)(v) + d^+(g_{\sigma}(n)) - \sum_{v \in V\setminus \{g_{\sigma}(n)\}} E(g_{\sigma}(n),v)) \\
        &= \sum_{v \in V} B_{\sigma}(n)(v)
    \end{align*} and the conclusion follows.
\end{proof}

Our final goal is therefore to show that $c = \min_{\sigma \in \mathcal{F}} h_{\sigma}$. By Theorem~\ref{First formula} we already know that $c = h^* \leq \min_{\sigma \in \mathcal{F}} h_{\sigma}$. Therefore it only remains to build a primitive function $\sigma$ such that $h_{\sigma} \leq h^*.$ Since $h^* = h_g$ for a certain optimal strategy $g$, we will fix such a $g$ and build $\sigma$ in function of $g$. The definition of $\sigma$ requires several preparatory lemmas.

\begin{lemma}\label{lemtech1}
Let $g$ be a strategy and assume there exist $v^*,w \in V$ and $r \in \mathbb{N}$ such that 
    \begin{enumerate}
        \item $E(w, v^*) \ge 1$.
        \item For all $n \in \mathbb{N}$ with $n \ge r$, $g(n) \ne v^*$.
        \item For all $s \in \mathbb{N}$, there exists $m \in \mathbb{N}$ with $m \ge s$ such that $g(m) = w$.
    \end{enumerate}
Then for all $t \in \mathbb{N}$ and all $m \geq r$, there exists $m' \ge m$ such that $$B_{g} (m')(v^*) \le \max(0,B_{g} (m)(v^*)-t) .$$
\end{lemma}
\begin{proof}
By induction of $t$. The case $t=0$ is trivial since one can choose $m'=m.$ Let us now assume that the result holds for a certain $t \in \N$ and fix $m \geq r.$ Thus we know there exists some $m'\geq m$ such that $$B_{g} (m')(v^*) \le \max(0,B_{g} (m)(v^*)-t) .$$ By $3.$ applied to $s=m',$ there exists $m^* \geq s$ such that $g(m^*)=w.$ Since $m^* \geq m' \geq m \geq r,$ we know by $2.$ that $g(n) \neq v^*$ for all $n \in [[m',m^*]].$ Hence by $1.$ and the induction hypothesis, we get 

\begin{align*}
            B_{g} (m^*+1)(v^*) &= \max( B_{g}( m^* )(v^*) - E(g(m^*),v^*),0 )\\
            &\le \max( B_{g}(m' )(v^*) - 1 ,0 )\\
            &\le \max( \max(0,B_{g} (m)(v^*)-t) - 1 ,0 )\\
            &= \max(0,B_{g} (m)(v^*)-t - 1),
        \end{align*}
hence the conclusion.
\end{proof}

\begin{lemma}\label{lemtech2}
Let $g$ be a strategy and assume there are $v^*,w \in V$ and $r \in \mathbb{N}$ such that 
    \begin{enumerate}
        \item $E(w, v^*) \ge 1$.
        \item For all $n \in \mathbb{N}$ with $n \ge r$, $g(n) \ne v^*$.
        \item For all $s \in \mathbb{N}$, there exists $m \in \mathbb{N}$ with $m \ge s$ such that $g(m) = w$.
    \end{enumerate}
Then for all $N \in \mathbb{N}$, there exists $m \in \mathbb{N}$ such that $m \ge r$ and $\sum_{v \in V} B_{g} (m)(v) \ge N $.
\end{lemma}
\begin{proof}
By induction on $N$. The initial case is straightforward since the dynamical bound is always positive. Let us assume that the result is verified for $N \in \N$ and thus that there exists some $m \geq r$ such that $\sum_{v \in V} B_{g} (m)(v) \ge N $. By applying the previous lemma with $t = B_g(m)(v^*)$ we deduce the existence of $m' \geq m$ such that $B_{g} (m')(v^*)=0.$ By $3.$ applied to $s=m'$, we know there exists some $m^* \geq m'$ such that $g(m^*)=w.$ Moreover, by $2.$, $g(n)\neq v^*$ for all $n \in [[m',m^*]]$. As a consequence, one has \[
    B_{g}(m^*)(v^*) \le B_{g}(m')(v^*) \le 0 < E(w,v^*).
    \] By the contraposition of Corollary~\ref{globalgrowth}, we get $$\sum_{v \in V} B_{g}(m^* + 1)(v) > \sum_{v \in V} B_{g}(m^*)(v)$$. Hence, using Corollary~\ref{globalgrowth} again, we get \begin{align*}
        \sum_{v \in V} B_{g}(m^* + 1)(v) &\ge \sum_{v \in V} B_{g}(m^*)(v) + 1\\
        &\ge \sum_{v \in V} B_{g}(m)(z) + 1\\
        &\ge N +1,
    \end{align*}
hence the conclusion, since $m^*+1 \geq r.$
\end{proof}

\begin{lemma}\label{lemtech3}
Let $g$ be an optimal strategy. For all $v \in V$ and $r \in \mathbb{N}$, there exists $n \ge r $ such that $g(n) = v$.
\end{lemma}
\begin{proof}
We prove this result by contradiction. Hence, suppose that there exist $v \in V$ and $r \in \N$ such that for all $n \geq r$ one has $g(n) \neq v.$ Therefore $g^{-1}(\{v\})$ is finite. On the other hand, since $g$ is defined on $\N$ with values in a finite set, there exists also a vertex $v'$ such that $g^{-1}(\{v'\})$ is infinite. Since $G$ is strongly connected, there is a path connecting $v$ and $v'$ and, on this path, it must be the case that two neighbors $v^*$ and $w$ are such that $g^{-1}(\{v^*\})$ is finite and $g^{-1}(\{w\})$ is infinite. For these two vertices, the three conditions of Lemma~\ref{lemtech2} are satisfied and we deduce that $$h_g = \lim_{m\to +\infty}\sum_{v \in V} B_{g} (m)(v) = +\infty.$$ But since $g$ is optimal by hypothesis, one must have $h_g = h^* = c < +\infty$ by Theorem~\ref{First formula}, thus a contradiction.
\end{proof}

From now on, we fix an optimal strategy $g$.

\begin{definition}
We set $$Z : V \times \N \to \mathcal{P}(\N), \quad (v,n) \mapsto \{j \in [[0,n]] : g(j) = v\}.$$ By Lemma~\ref{lemtech3}, the set $\{ n \in \mathbb{N} : |Z(v,n)| = m\}$ is non-empty for all $m \in \N_+$ and thus we may define\[
    z : V \times \mathbb{N}_+ \rightarrow \mathbb{N}, \quad  (v,m) \mapsto \min \{ n \in \mathbb{N} : |Z(v,n)| = m\}.
    \] In other words, $z(v,m)$ gives the generation at which $v$ is activated for the $m^{\text{th}}$ time.
\end{definition}

\begin{lemma}\label{zproperties}
The functions $Z$ and $z$ have the following properties:

\begin{enumerate}
\item[$\bullet$] $Z$ is increasing on the second component.
\item[$\bullet$] For all $m \in \N_+$ and $v \in V$, we have $|Z(v,z(v,m))|=m$.
\item[$\bullet$] For all $m \in \N_+$ and $v \in V$, we have $g(z(v,m))=v.$
\item[$\bullet$] $z$ is injective.
\end{enumerate}
\end{lemma}
\begin{proof}
The first two properties are trivial. For the third one, let $n=z(v,m).$ If $n=0$, the result is clear since $m > 0$ and $|Z(v,n)| = m.$ If $n \geq 1$, suppose by contradiction that $g(n) \neq v$. Hence \[
    |Z(v,n-1)| = |\{j \in [[0,n-1]] : g(j) = v\}| = |\{j \in [[0,n]] : g(j) = v\}| = |Z(v,n)| = m,
    \] which contradicts the minimality of $n$. Finally, to prove the fourth property, let $v,v' \in V$ and $m,m' \in \N_+$ be such that $z(v,m)=z(v',m').$ By the third property, one gets $$v=g(z(v,m))=g(z(v',m'))=v'.$$ Hence, by the second property, one also gets $$m=Z(v,z(v,m))=Z(v',z(v',m'))=m',$$ which proves the injectivity of $z$.
\end{proof}

We can finally define $\sigma$ in function of $g$.

\begin{definition}
Let us set \[
    W = \bigcup_{i = 1}^N \bigcup_{m=1}^{v_G(i)} \{z(v_i,m)\}.
    \]
    By the injectivity of $z$, we know that $|W| = \sum_{i=1}^N v_G(i) = P$. Let then $p : [[0,P-1]] \rightarrow W $ be an ordering of $W$ (i.e. $p(i) < p(j)$ if $i < j$) and define $\sigma$ by
    \[
    \sigma : [[0,P-1]] \rightarrow V, \quad k \mapsto g(p(k)).
    \]
\end{definition}

\begin{remark}
Notice that $p(0)=0.$ Indeed, $0 \in W$ since $0 = z(g(0),1).$
\end{remark}

\begin{proposition}
$\sigma$ is a primitive function.
\end{proposition}
\begin{proof}
Let $i \in [[1,N]]$. Using Lemma~\ref{zproperties}, we have
        \begin{align*}
            |\sigma^{-1}(\{v_i\}) | &= |\{k \in [[0,P-1]] : g(p(k)) = v_i\}|\\
            &= |\{j \in W : g(j) = v_i\}| \\
            &= |W \cap g^{-1}(\{v_i\}) |\\
            &= |\bigcup_{i' = 1}^N \bigcup_{m=1}^{v_G(i')} \{z(v_{i'},m)\}\cap g^{-1}(\{v_i\})|\\
            &= \sum_{i' = 1}^N \sum_{m=1}^{v_G(i')} \mathds{1}_{g(z(v_{i'},m)) = v_i} \\
            & = \sum_{i' = 1}^N \sum_{m=1}^{v_G(i')} \mathds{1}_{v_{i'} = v_i}\\
            &= v_G(i),
        \end{align*} hence the conclusion.
\end{proof}

\begin{definition}
We set \[
    Z_{\sigma} : V \times \mathbb{N} \rightarrow \mathcal{P}(\mathbb{N}), \quad (v,n) \mapsto \{ j \in [[0,n]] : \sigma(\overline{j}) = v \}
    \]
    and for all $k \in [[0,P-1]] $
    \[
    W_k =  [[0,p(k)-1]] \cap W, \quad R_k  = [[0,p(k)-1]] \setminus W_k.
    \]
In the above definition, if $p(k) = 0$, we use the convention $ [[0,-1]] = \emptyset $. Finally, for all $k \in [[0,P-1]] $, we set
    \[
    V_k = g(R_k) = \{v \in V: \exists j \in R_k \textup{ with } g(j) = v\}.
    \]
\end{definition}

\begin{lemma}\label{lemtech4}
Let $k \in [[0,P-1]] $ and $ v_i \in V $ such that $|Z(v_{i}, p(k))| \ge v_G(i)$. Then $|Z_{\sigma}(v_i, k)| \ge v_G(i)$.
\end{lemma}
\begin{proof}
Let
        \[
        S_i = \bigcup_{m=1}^{v_G(i)} \{z(v_i,m)\}.
        \]
It is clear that $S_i \subseteq W$ and that $|S_i| = v_G(i)$. Moreover, using Lemma~\ref{zproperties}, for all $j \in S_i$ one has $g(j) = v_i$ and $|Z(v_i, j)| \le v_G(i) \le |Z(v_{i}, p(k))|$. We deduce that $j \le p(k)$ for all $j \in S_i$. Let us now set $T_i = p^{-1}(S_i)$. Since $p$ is an increasing bijection, we get that $|T_i| = v_G(i)$ as well as $l \le k \le P-1$ for all $l \in T_i$. Furthermore, for all $l \in T_i$, there exists $j \in S_i$ such that $p(l) = j$. Hence we have
        \[
        \sigma(\overline{l}) = \sigma(l) = g(p(l)) = g(j) = v_i.
        \]
and therefore
        \[
            v_G(i) = |T_i| \le |\{ l \in [[0,k]] : \sigma(\overline{l}) = v_i\}| = |Z_{\sigma}(v_i, k)|.
       \]
Hence the conclusion.
\end{proof}

\begin{lemma}
Let $ k \in [[1,P-1]], v_i \in V$ and $j \in [[p(k-1),p(k)-1]]$. If $|Z(v_{i},j)| \ge v_G(i) $, then $ |Z(v_{i},p(k-1))| \ge v_G(i) $.
\end{lemma}
\begin{proof}
By contradiction, assume that $|Z(v_{i},p(k-1))| < v_G(i) $ and set $j^* = z(v_i, v_G(i))$. By Lemma~\ref{zproperties} we know that $g(j^*) = v_i$ and that $|Z(v_{i},j^*)| = v_G(i) \leq  |Z(v_{i},j)|$. Hence
        \[
        p(k-1) < j^* \le j < p(k).
        \]
Since $j^* \in W$, there exists $r^* \in [[0,P-1]] $ such that $p(r^*) = j^*$ and since $p$ is strictly increasing, that means that $k-1 < r^*< k$, hence a contradiction.
\end{proof}

\begin{lemma}\label{invk}
Let $ k \in [[1,P-1]]$. If $v_{\iota} \in V_k$, then $ B_{\sigma}(k)(v_{\iota}) \ge E(g_{\sigma}(k), v_{\iota})  $.
\end{lemma}
\begin{proof}
The idea is to apply Lemma~\ref{suffgood} and, for that purpose, it is enough to show that if $v_{\iota} \in V_k$ then $|Z_{\sigma}(v_{\iota},k-1)| \geq v_G(\iota).$ Let then $v_{\iota} \in V_k.$ By definition, there exists $j \in R_k$ such that $g(j) = v_{\iota}$. Since $R_k \subseteq [[0,p(k)-1]]\backslash W$, one has $j \notin W$ and thus
        \[
        j \notin \bigcup_{i = 1}^N \bigcup_{m=1}^{v_G(i)} \{z(v_i,m)\}.
        \]
Now, assume by contradiction that $|Z(v_{\iota}, j)| \le v_G(\iota)$ and set $m = |Z(v_{\iota}, j)|$. That implies that $z(v_{\iota},m)=j$ with $m \le v_G(\iota)$, which contradicts $j \notin W$. Hence $|Z(v_{\iota}, j)| > v_G(\iota)$ with $j \le p(k) - 1$. If $p(k-1) \le j$, then by the previous lemma, $|Z(v_{\iota}, p(k-1))| \ge v_G(\iota)$ and one directly gets the conclusion by Lemma~\ref{lemtech4}. Otherwise, if $j < p(k-1)$, then we also have $|Z(v_{\iota}, p(k-1))| \ge v_G(\iota)$, this time using the fact that $Z$ is increasing on the second component. Again by Lemma~\ref{lemtech4}, we obtain $|Z_{\sigma}(v_{\iota}, k-1)| \ge v_G(\iota)$ and the conclusion follows.
\end{proof}

\begin{lemma}
Let $k \in [[1, P-1]] $ and $v \notin V_k $. Then $B_{g}(p(k-1)+1)(v) \ge B_{g}(p(k) )(v) $.
\end{lemma}
\begin{proof}
We shall actually prove that, for all $j \in [[p(k-1)+1, p(k)-1]] $, one has $$B_{g}(j)(v) \ge B_{g}(j+1)(v),$$ which implies the thesis. Let then $j \in [[p(k-1)+1, p(k)-1]]$. Since $p(k-1) < j < p(k) $ and since $p$ is a strictly increasing bijection, we deduce that $j \in R_k$. Since $v \notin V_k $, this implies that $g(j) \ne v$. Therefore
        \[
        B_{g}(j+1)(v) = \max( B_{g}(j)(v) - E(g(j),v),0 ) \le B_{g}(j)(v)
        \]
as announced.
\end{proof}

\begin{lemma}\label{notinvk}
Let $k \in [[0,P-1]] $ and $v \notin V_k $. Then $B_{g}(p(k))(v) \le B_{\sigma}(k)(v)$.
\end{lemma}
\begin{proof}
By induction on $k$. The case $k = 0$ is clear since $p(0) = 0$, and
        \[
        B_{g}(p(0))(v) = 0 = B_{\sigma}(0)(v).
        \]
Now let $k \in [[1,P-1]] $ and assume that the result holds for $k-1$. Using the previous lemma, the equality $g(p(k-1))=\sigma(k-1)=\sigma(\overline{k-1})=g_{\sigma}(k-1)$ and the induction hypothesis, one gets
        \begin{align*}
            B_{g}(p(k))(v) &\le B_{g}(p(k-1)+1)(v)\\
            &=\begin{cases}
                B_{g}(p(k-1))(v) + d^+(v) & \textup{if } g(p(k-1)) = v\\
                \max(B_{g}(p(k-1))(v)- E(g(p(k-1)),v), 0) & \textup{if } g(p(k-1)) \ne v\\
            \end{cases}   \\
            &\le \begin{cases}
                B_{\sigma}(k-1)(v) + d^+(v) & \textup{if } g(p(k-1)) = v\\
                \max(B_{\sigma}(k-1)(v)- E(g(p(k-1)),v), 0) & \textup{if } g(p(k-1)) \ne v\\
            \end{cases}   \\
            &= \begin{cases}
                B_{\sigma}(k-1)(v) + d^+(v) & \textup{if } g_{\sigma}(k-1) = v\\
                \max(B_{\sigma}(k-1)(v)- E(g_{\sigma}(k-1),v), 0) & \textup{if } g_{\sigma}(k-1) \ne v\\
            \end{cases}\\
            &= B_{\sigma}(k)(v).
        \end{align*}
Hence the conclusion.
\end{proof}

\begin{lemma}
For all $k \in [[0,P-1]] $, one has $$\left(\sum_{v \in V} B_{g}(p(k)+1)(v) \right) -  \left(\sum_{v \in V} B_{g}(p(k))(v) \right) \ge \left(\sum_{v \in V} B_{\sigma}(k+1)(v) \right) -  \left(\sum_{v \in V} B_{\sigma}(k)(v) \right) $$
\end{lemma}
\begin{proof}
The case $k = 0$ is immediate. Indeed, since $p(0) = 0$, we have
        \[
        \sum_{v \in V} B_{g}(p(0))(v) = \sum_{v \in V} B_{\sigma}(0)(v) = 0
        \]
        and therefore it is enough to show that
        \[
        \sum_{v \in V} B_{g}(1)(v) \ge \sum_{v \in V} B_{\sigma}(1)(v).
        \]
        This is obviously the case because
        \begin{align*}
            B_{g}(1)(v) &= \begin{cases}
                d^+(v) &\textup{if } v= g(0)\\
                0 &\textup{otherwise}
            \end{cases}\\
            &= \begin{cases}
                d^+(v) &\textup{if } v= g_{\sigma}(0)\\
                0 &\textup{otherwise}
            \end{cases}\\
            &= B_{\sigma}(1)(v).
        \end{align*}
Let us now prove the result for $k \in [[1,P-1]] $. By Proposition~\ref{delta}, we have to show that 
        \[
        \sum_{v \in V} \delta_{g}(v, p(k)) \ge \sum_{v \in V} \delta_{g_{\sigma}}(v,k).
        \]
        If $v \in V_k$, then $ B_{\sigma}(k)(v) \ge E(g_{\sigma}(k), v)  $ by Lemma~\ref{invk} and thus
        \[
        \delta_{g_{\sigma}}(v,k) = \max(0 , E(g_{\sigma}(k), v) - B_{\sigma}(k)(v)) = 0 \le \delta_{g}(v, p(k)).
        \]
        If $v \notin V_k $, then $B_{g}(p(k))(v) \le B_{\sigma}(k)(v)$ by Lemma~\ref{notinvk} and thus
        \[
        \delta_{g_{\sigma}}(v,k) = \max(0 , E(g_{\sigma}(k), v) - B_{\sigma}(k)(v)) \le  \max(0 , E(g(p(k)), v) - B_{g}(p(k))(v)) = \delta_{g}(v, p(k)).
        \]
The conclusion follows directly.
    \end{proof}
    
We can finally prove the main theorem of this section.

\begin{proof}[Proof of Theorem~\ref{Second formula}]
As already explained, it only remains to show that $h_{\sigma} \leq h^*.$ We first show by induction on $k \in [[0,P-1]]$ that 
    \[
    \sum_{v \in V} B_{\sigma}(k)(v) \le \sum_{v \in V} B_{g}(p(k))(v).
    \]
The initial case is trivial since $p(0)=0.$ Let us now assume that the result holds for $k\in [[0,P-2]]$. Using the previous lemma and Corollary~\ref{globalgrowth}, we obtain  \begin{align*}
        \sum_{v \in V} B_{g}(p(k+1))(v) &\ge \sum_{v \in V} B_{g}(p(k) + 1)(v)\\
        &\ge \left(\sum_{v \in V} B_{g}(p(k))(v) \right) + \left(\sum_{v \in V} B_{\sigma}(k+1)(v) \right) -  \left(\sum_{v \in V} B_{\sigma}(k)(v) \right)\\
        &\ge \sum_{v \in V} B_{\sigma}(k+1)(v)
    \end{align*}
which settles this induction step. We can now write this inequality with $k=P-1$ to get \[
    h_{\sigma} = \sum_{v \in V} B_{\sigma}(P-1)(v) \le \sum_{v \in V} B_{g}(p(P-1))(v) \le h^*,
    \] 
hence the conclusion.
\end{proof}

\begin{remark}
The reader may notice that, in these various proofs, we only use the definition of $v_G$ at one crucial moment, i.e. in the proof of Lemma~\ref{suffgood}. Moreover, we only use the fact that $v_G$ cancels $L$, not its minimality. The minimality's purpose is obviously to have a canonical vector and to make the formula $\min_{\sigma \in \mathcal{F}} \left(\sum_{v \in V} B_{\sigma}(P-1)(v) \right)$ as quick to compute as possible. Unfortunately, $P$ may be exponentially large (see for instance the family of graphs build in the proof of Theorem $2$ in \cite{Vanp15}) and is generally bounded by $(2D)^{N-1}$ (see \cite{Bjor92}, Proposition 4.10.).
\end{remark} 

\section{Primitive extension and primitive feedback number}

The main goal of this section is to rework the equality $c=\min_{\sigma \in \mathcal{F}} \left(\sum_{v \in V} B_{\sigma}(P-1)(v) \right)$ in order to write $c$ as a more "intrinsic property" of the graph and, thereby, equalizing it to a "primitive" version of the feedback number. 

\begin{definition}\label{prim}
Let $\hat{V}$ be a set of cardinality $P$ whose elements are denoted by $v_{i,j}$ for $1 \le i \le N$ and $1\le j \le v_G(i)$. The primitive extension of $G$ over $\hat{V}$ or simply \textit{the primitive extension of $G$}, is the graph $\hat{G} = (\hat{V}, \hat{E})$ defined by $\hat{E}(v_{i,j}, v_{k,l}) = E(v_i,v_k)$ for all $i,j,k,l.$
Moreover we define the function \[
    \xi : \hat{V} \rightarrow V, \quad v_{i,j} \mapsto v_i.
    \]
\end{definition} 

From now we will simultaneously work with several different graphs, hence we will use subscript notations, e.g. $d_G^+(v)$ instead of $d^+(v).$

\begin{lemma}\label{weaklemma}
Let $w_0, \dots, w_{P-1}$ be an ordering on $\hat{V}$ and $\tilde{G} = (\hat{V}, \tilde{E})$ a loop-free graph such that 
\begin{itemize}
        \item [$(1)$] For all $i < j$, one has $\tilde{E}(w_i, w_j) = 0.$
        \item [$(2)$] For all $m \in [[1,P-1]] $ and $v \in V$, if $\sum_{w \in \hat{V}} \mathds{1}_{\xi(w) = v} \tilde{E}(w_m, w) < E(\xi(w_m), v) $, then for all $i \in [[0,m-1]] $ such that $\xi(w_i) = v $ one has
        \[
        d^+_G(v) - \sum_{j=0}^{m}\tilde{E}(w_j, w_i) \le 0.
        \]
\end{itemize}
Let $\sigma \in \mathcal{F}$ be the primitive function defined by
    \[
    \sigma : [[0,P-1]] \rightarrow V, \quad m \mapsto \xi(w_m) .
    \]
For all $v_k \in V$, we then have
    \[
    B_{\sigma}(P)(v_k) \ge d^+_G(v_k)*v_G(k) - \sum_{l=1}^{v_G(k)}d^{-}_{\tilde{G}}(v_{k,l}).
    \]
\end{lemma}
\begin{proof}
We shall first prove by induction on $m \in [[0,P-1]] $ that 
    \[
    B_{\sigma}(m+1)(v_k) \ge \sum_{i = 0}^{m} \mathds{1}_{\xi(w_i) = v_k} \left( d^+_{G}(v_k) - \sum_{j=0}^{m} \tilde{E}(w_j, w_i) \right). \quad (*)
    \]
The initial case $m = 0$ is clear. Indeed, if $\xi(w_0) = v_k$, then
    \[
    B_{\sigma}(1)(v_k) = d^+_G(v_k) = \mathds{1}_{\xi(w_0) = v_k} \left( d^+_{G}(v_k) - \tilde{E}(w_0, w_0) \right)
    \]
and if $\xi(w_0) \neq v_k$, then
    \[
    B_{\sigma}(1)(v_k) = 0 = \mathds{1}_{\xi(w_0) = v_k} \left( d^+_{G}(v_k) - \tilde{E}(w_0, w_0) \right).
    \]
Let us now assume that the result holds for $m\in [[0,P-2]]$. In order to prove it for $m+1$, we distinguish three cases.

\medskip
\noindent
\textbf{First case}: $v_k \ne \xi(w_{m+1})$ and $\sum_{w \in \hat{V}} \mathds{1}_{\xi(w) = v} \tilde{E}(w_{m+1}, w) \ge E(\xi(w_{m+1}), v).$

\medskip
\noindent
In this case, we have

\begin{align*}
        B_{\sigma}(m+2)(v_k) &\ge B_{\sigma}(m+1)(v_k) - E(g_{\sigma}(m+1), v_k)\\
        &\ge \left( \sum_{i = 0}^{m} \mathds{1}_{\xi(w_i) = v_k} \left( d^+_{G}(v_k) - \sum_{j=0}^{m} \tilde{E}(w_j, w_i) \right) \right) - E(\xi(w_{m+1}), v_k )\\
        &\ge \left( \sum_{i = 0}^{m} \mathds{1}_{\xi(w_i) = v_k} \left( d^+_{G}(v_k) - \sum_{j=0}^{m} \tilde{E}(w_j, w_i) \right) \right) - \sum_{w \in \hat{V}} \mathds{1}_{\xi(w) = v_k} \tilde{E}(w_{m+1}, w)\\
        &\underset{(1)}= \left( \sum_{i = 0}^{m} \mathds{1}_{\xi(w_i) = v_k} \left( d^+_{G}(v_k) - \sum_{j=0}^{m} \tilde{E}(w_j, w_i) \right) \right) - \sum_{i = 0}^{m} \mathds{1}_{\xi(w_i) = v_k} \tilde{E}(w_{m+1}, w_i)\\
        &=  \sum_{i = 0}^{m} \mathds{1}_{\xi(w_i) = v_k} \left( d^+_{G}(v_k) - \sum_{j=0}^{m+1} \tilde{E}(w_j, w_i) \right) = \sum_{i = 0}^{m+1} \mathds{1}_{\xi(w_i) = v_k} \left( d^+_{G}(v_k) - \sum_{j=0}^{m+1} \tilde{E}(w_j, w_i) \right).
    \end{align*}
\noindent   
\textbf{Second case}: $v_k \ne \xi(w_{m+1})$ and $\sum_{w \in \hat{V}} \mathds{1}_{\xi(w) = v} \tilde{E}(w_{m+1}, w) < E(\xi(w_{m+1}), v).$

\medskip
\noindent
In this case, using $(2)$, we have  

 \begin{align*}
        B_{\sigma}(m+2)(v_k) &\ge 0\\
        &\ge \sum_{i = 0}^{m} \mathds{1}_{\xi(w_i) = v_k} \left( d^+_{G}(v_k) - \sum_{j=0}^{m+1} \tilde{E}(w_j, w_i) \right)\\
        &= \sum_{i = 0}^{m+1} \mathds{1}_{\xi(w_i) = v_k} \left( d^+_{G}(v_k) - \sum_{j=0}^{m+1} \tilde{E}(w_j, w_i) \right).
    \end{align*} 
\noindent   
\textbf{Third case}: $v_k = \xi(w_{m+1})$.

\medskip
\noindent
In this case, we have

\begin{align*}
        B_{\sigma}(m+2)(v_k) &= B_{\sigma}(m+1)(v_k) + d^+_G(v_k)\\
        &\ge \left( \sum_{i = 0}^{m} \mathds{1}_{\xi(w_i) = v_k} \left( d^+_{G}(v_k) - \sum_{j=0}^{m} \tilde{E}(w_j, w_i) \right) \right) + d^+_G(v_k)\\
        &\ge \left( \sum_{i = 0}^{m} \mathds{1}_{\xi(w_i) = v_k} \left( d^+_{G}(v_k) - \sum_{j=0}^{m+1} \tilde{E}(w_j, w_i) \right) \right) + d^+_G(v_k)\\
        &\ge \sum_{i = 0}^{m+1} \mathds{1}_{\xi(w_i) = v_k} \left( d^+_{G}(v_k) - \sum_{j=0}^{m+1} \tilde{E}(w_j, w_i) \right).
    \end{align*}
This concludes the induction step. We can now apply $(*)$ to $m=P-1$ to get 
\begin{align*}
        B_{\sigma}(P)(v_k) &\ge \sum_{i = 0}^{P-1} \mathds{1}_{\xi(w_i) = v_k} \left( d^+_{G}(v_k) - \sum_{j=0}^{P-1} \tilde{E}(w_j, w_i) \right)\\
        &= \sum_{l=1}^{v_G(k)} \left( d^+_{G}(v_k) - \sum_{w \in \hat{V}} \tilde{E}(w, v_{k,l}) \right)\\
        &= d^+_G(v_k)*v_G(k) - \sum_{l=1}^{v_G(k)}d^{-}_{\tilde{G}}(v_{k,l}).
    \end{align*}
Hence the conclusion.
\end{proof}

We shall immediately reinforce this lemma under further conditions on $\tilde{G}$.

\begin{lemma}\label{stronglemma}
Let $w_0, \dots, w_{P-1}$ be an ordering on $\hat{V}$ and $\tilde{G} = (\hat{V}, \tilde{E})$ a loop-free graph such that
    \begin{itemize}
        \item [$(1)$] For all $i < j$, one has $\tilde{E}(w_i, w_j) = 0.$
        \item [$(2)$] For all $v \in V$ and $m \in [[0,P-1]] $, $\sum_{w \in \hat{V}} \mathds{1}_{\xi(w) = v} \tilde{E}(w_m, w) \le E(\xi(w_m), v) $.
        \item [$(3)$] For all $m \in [[0,P-1]] , d_G^+(\xi(w_m)) \ge d^{-}_{\tilde{G}}(w_m).$
        \item [$(4)$] $\tilde{G}$ is maximal, in the sense that for any loop-free graph $G' = (\hat{V}, E')$ satisfying conditions $(1) $ to $ (3)$, we have
        $
        M' = \sum_{w \in \hat{V}} d^+_{G'}(w) \le \sum_{w \in \hat{V}} d^+_{\tilde{G}}(w) = \tilde{M}.
        $
    \end{itemize}

\medskip
\noindent
Let $\sigma \in \mathcal{F}$ be the primitive function defined by
    \[
    \sigma : [[0,P-1]] \rightarrow V, \quad m \mapsto \xi(w_m) .
    \]
    For all $v_k \in V$, we then have
    \[
    B_{\sigma}(P)(v_k) \ge d^+_G(v_k)*v_G(k) - \sum_{l=1}^{v_G(k)}d^{-}_{\tilde{G}}(v_{k,l}).
    \]
\end{lemma}
\begin{proof}
Let $\mathcal{T}$ be the function defined by
    \[
    \mathcal{T} : \mathbb{N}^{\hat{V} \times \hat{V}} \rightarrow \mathbb{Z} : \tilde{E} \mapsto \sum_{i=0}^{P-1}\sum_{j=0}^{P-1} \tilde{E}(w_i,w_j)(i-j).
    \]
We shall prove by induction on $M = \mathcal{T}(\tilde{E})$ that if  $\tilde{G}=(\hat{V},\tilde{E})$ satisfies the conditions $(1) $ to $(4)$, then for all $v_k \in V$,
    \[
    B_{\sigma}(P)(v_k) \ge d^+_G(v_k)*v_G(k) - \sum_{l=1}^{v_G(k)}d^{-}_{\tilde{G}}(v_{k,l}).
    \]
Note that if $\tilde{G}$ satisfies $(1)$, then $M \ge 0$. Hence the initial case is $M=0.$ Let us prove this initial case and let $\tilde{G}=(\hat{V},\tilde{E})$ be a loop-free graph such that $\mathcal{T}(\tilde{E}) = 0.$ It implies that $\tilde{E}(w_i,w_j) = 0$ for all $i,j \in [[0,P-1]].$ Therefore, by the conditions $(2)$ and $(3)$, and the maximality $(4)$ of $\tilde{G}$, for all $i,j \in [[0,P-1]] $ with $i > j$, either $E(\xi(w_i), \xi(w_j)) = 0$ or $d^+_G(\xi(w_j)) = 0$. Indeed, otherwise we could add an edge from $w_i$ to $w_j$ while preserving conditions $(1)$ to $(3)$, contradicting the maximality. Consider $v = \xi(w_{0}).$ On the one hand, if $d^+_G(v) = 0$, then $V$ is a singleton since $G$ is strongly connected. On the other hand, if $E(\xi(w_m),v) = 0$ for all $m \in [[1,P-1]]$, then $d^-_G(v) = 0$, since $E(\xi(w_0),v)=E(v,v)=0$ and we can also conclude that $V$ is a singleton. Therefore it is clear that $P=1$ and $\hat{V}= \{v_{1,1}\}.$ We immediately get \[
    B_{\sigma}(P)(v) = B_{\sigma}(1)(v) = d^+_G(v) = 0 = d^+_G(v)*v_G(1) - \sum_{l=1}^{v_G(1)}d^{-}_{\tilde{G}}(v_{1,l}),
    \] which proves the initial case.
    
\medskip
\noindent
Let us now fix a loop-free graph $\tilde{G}=(\hat{V},\tilde{E})$, verifying the conditions $(1)$ to $(4)$, such that $ \mathcal{T}(\tilde{E}) > 0$ and assume that the result holds for all loop-free graphs $G'=(\hat{V},E')$, verifying the same conditions, such that $\mathcal{T}(E') < \mathcal{T}(\tilde{E}).$ We shall prove that it holds for $\tilde{G}$ as well. If $\tilde{G}$ satisfies the second condition of Lemma~\ref{weaklemma}, then we get the conclusion by applying this lemma. Otherwise, there exists $m^* \in [[1,P-1]] $, $i^* \in [[0,m^*-1]]$ and $v^* \in V$ such that $\sum_{w \in \hat{V}} \mathds{1}_{\xi(w) = v^*} \tilde{E}(w_{m^*}, w) < E(\xi(w_{m^*}), v^*) $, $\xi(w_{i^*}) = v^*$ and
    $
    d^+_G(v^*) - \sum_{j=0}^{m^*}\tilde{E}(w_j, w_{i^*}) > 0.
    $ 
We can note that $\xi(w_{m^*}) \ne v^*$ since we have
    $
    \sum_{w \in \hat{V}} \mathds{1}_{\xi(w) = v^*} \tilde{E}(w_{m^*}, w) < E(\xi(w_{m^*}), v^*).
    $
We shall now show by contradiction that there exists $r \in [[m^*+1, P-1]] $ such that $\tilde{E}(w_r, w_{i^*}) > 0$. Thus, let us assume that $\tilde{E}(w_r, w_{i^*}) = 0$ for all $r \in [[m^*+1, P-1]] $. We obtain
\[
d^+_G(v^*) - d^-_{\tilde{G}}(w_{i^*}) = d^+_G(v^*) - \sum_{j=0}^{m^*}\tilde{E}(w_j, w_{i^*}) > 0.
\]
Let us then define the loop-free graph $G' = (\hat{V}, E')$ by
    \[
    E' : \hat{V} \times \hat{V} \rightarrow \mathbb{N}, \quad (w_i,w_j) \mapsto \tilde{E}(w_i,w_j) + \mathds{1}_{(i,j) = (m^*,i^*) }.
    \]
We shall show that $G'$ satisfies the conditions $(1)$ to $ (3)$. Property $(1)$ is immediate. Indeed, if $i < j$, then $(i,j) \ne (m^*,i^*)$ and 
\[
    E'(w_i, w_j) = \tilde{E}(w_i,w_j) + \mathds{1}_{(i,j) = (m^*,i^*) }  = \tilde{E}(w_i,w_j) = 0.
    \]
Next, we prove Property $(2)$. For all $m \in [[1,P-1]] $ and $ v \in V$, we have 
\[
E(\xi(w_m), v) \ge \sum_{w \in \hat{V}} \mathds{1}_{\xi(w) = v} \tilde{E}(w_m, w) = \sum_{w \in \hat{V}} \mathds{1}_{\xi(w) = v} E'(w_m, w) - \mathds{1}_{m=m^*, v=v^*},
\]
which proves Property $(2)$ for $G'$ unless $m=m^*$ and $v=v^*$. But in that case, by assumption, we have 
\[
\sum_{w \in \hat{V}} \mathds{1}_{\xi(w) = v^*} \tilde{E}(w_{m^*}, w) < E(\xi(w_{m^*}), v^*),
\]
which means that Property $(2)$ is still verified for $G'$ in that case. We now prove Property $(3)$. If we let $m \in [[0,P-1]] $, we have
\[
d^-_{G'}(w_m) = d^-_{\tilde{G}}(w_m) + \mathds{1}_{m=i^*} \le d_G^+(\xi(w_m)) + \mathds{1}_{m=i^*},
\]
which proves property $(3)$ unless $m = i^*$. If $m=i^*$, then since $d_G^+(\xi(w_{i^*})) - d^{-}_{\tilde{G}}(w_{i^*}) > 0  $, Property $(3)$ is verified in that case as well. We have shown that $G'$ satisfies conditions $(1)$ to $(3)$. Since the number of edges of $G'$ is strictly greater than the number of edges of $\tilde{G}$, we obtain the required contradiction as $\tilde{G}$ is supposed to be maximal by property $(4)$. We thus have shown that there exists $r \in [[m^*+1, P-1]] $ such that $\tilde{E}(w_r, w_{i^*}) > 0$. 

\medskip
\noindent
Now, we define $G' = (\hat{V}, E')$ by
    \[
    E' : \hat{V} \times \hat{V} \rightarrow \mathbb{N}, \quad (w_i,w_j) \mapsto \tilde{E}(w_i,w_j) + \mathds{1}_{(i,j) = (m^*,i^*) } - \mathds{1}_{(i,j)=(r,i^*)  }.
    \]
This time, we want to show that $G'$ is a loop-free graph which satisfies conditions $(1)$ to $(4)$ and has $\mathcal{T}(E') < \mathcal{T}(\tilde{E})$.
The proof of the conditions $(1)$ to $(3)$ is identical as previously, if only slightly simpler. Since $G'$ satisfies the conditions $(1)$ to $(3)$ and has the same number of edges as $\tilde{G}$, it is clear that $G'$ must be maximal as well and thus satisfies Property $(4)$. We now need to show that $\mathcal{T}(E') < \mathcal{T}(\tilde{E})$. This follows from
\begin{align*}
    \mathcal{T}(E') &= \sum_{i=0}^{P-1}\sum_{j=0}^{P-1} E'(w_i,w_j)(i-j) = \sum_{i=0}^{P-1}\sum_{j=0}^{P-1} \left(\tilde{E}(w_i,w_j) + \mathds{1}_{(i,j) = (m^*,i^*) } - \mathds{1}_{(i,j)=(r,i^*)} \right)(i-j)\\
    &= \mathcal{T}(\tilde{E}) + (m^* - i^*) - (r - i^*) = \mathcal{T}(\tilde{E}) + m^* - r <\mathcal{T}(\tilde{E}).
\end{align*}
We can now apply the induction hypothesis on $G'$ to deduce that for all $v_k \in V$, we have
\begin{align*}
    B_{\sigma}(P)(v_k) &\ge d^+_{G}(v_k) * v_G(k) - \sum_{l=1}^{v_G(k)} d^{-}_{G'}(v_{k,l}) = d^+_{G}(v_k) * v_G(k) - \sum_{l=1}^{v_G(k)} d^{-}_{\tilde{G}}(v_{k,l}).
\end{align*}
Hence the conclusion.
\end{proof}

\begin{theorem}\label{Genfeednum}
One has
    \[
    c =  \left(\sum_{i=1}^N d^+_G(v_i) * v_G(i) \right) - a,
    \]
where is $a$ is the biggest number of edges of $\hat{G}$ that one can keep so that the associated subgraph $\tilde{G} = (\hat{V}, \tilde{E})$ respects the three following conditions: 
    \begin{itemize}
        \item [$\bullet$] $\tilde{G}$ is acyclic.
        \item [$\bullet$] For all $v \in V$ and $w \in \hat{V}$, one has $\sum_{w' \in \hat{V}} \mathds{1}_{\xi(w') = v} \tilde{E}(w,w') \le E(\xi(w), v).$
        \item [$\bullet$] For all $w \in \hat{V}$, one has $d^{-}_{\tilde{G}}(w) \le d^+_G(\xi(w)).$
    \end{itemize}
\end{theorem}

\begin{proof}
We begin by showing that $c \le \left(\sum_{i=1}^N d^+_G(v_i) * v_G(i) \right) - a$. This is the "easiest" inequality since it does not required the previous technical lemma. Let $\tilde{G} = (\hat{V}, \tilde{E})$ be the subgraph of $\hat{G}$ associated to $a$, so that it satisfies the three conditions of the theorem and $a = \sum_{v \in \hat{V}} d^{-}_{\tilde{G}}(v)$. Since $\tilde{G}$ is acyclic, there is an ordering $w_0, \dots, w_{P-1}$ of $\hat{V}$ such that $d^+_{\tilde{G}}(w_0) = 0$ and $\tilde{E}(w_j, w_i) = 0$ for all $j < i$. Moreover, we define as usual the primitive function $\sigma \in \mathcal{F}$ by
    \[
    \sigma : [[0,P-1]] \rightarrow V, \quad m \mapsto \xi(w_m) .
    \]
We shall prove by induction on $m \in [[0,P-1]] $ that
    \[
    B_{\sigma}(m+1)(v) \le \sum_{i = 0}^{m} \mathds{1}_{\xi(w_i) = v} \left( d^+_{G}(v) - \sum_{j=0}^{m} \tilde{E}(w_j, w_i) \right), \quad (*)
    \]
for all $v \in V.$ The initial case $m=0$ is trivial. Indeed, if $v \ne \xi(w_0)$,
    \[
    B_{\sigma}(1)(v) = 0 = \mathds{1}_{\xi(w_0) = v} \left( d^+_{G}(v) - \tilde{E}(w_0, w_0) \right)
    \]
and if $v =\xi(w_0),$
    \[
    B_{\sigma}(1)(v) = d^+_G(v) = \mathds{1}_{\xi(w_0) = v} \left( d^+_{G}(v) - \tilde{E}(w_0, w_0) \right).
    \]
For the induction step, let us assume that the result holds for $m$ and prove it for $m+1$ by distinguishing three cases.

\medskip
\noindent
\textbf{Case 1:} $v \ne \xi(w_{m+1})$ and $\max(B_{\sigma}(m+1)(v)  , E(g_{\sigma}(m+1),v)) = B_{\sigma}(m+1)(v)$.

\medskip
\noindent
We have 
\begin{align*}
        B_{\sigma}(m+2)(v) &= B_{\sigma}(m+1)(v) - E(g_{\sigma}(m+1),v)\\
        &\le \left( \sum_{i = 0}^{m} \mathds{1}_{\xi(w_i) = v} \left( d^+_{G}(v) - \sum_{j=0}^{m} \tilde{E}(w_j, w_i) \right)\right) - E(g_{\sigma}(m+1),v)\\
        &= \left( \sum_{i = 0}^{m+1} \mathds{1}_{\xi(w_i) = v} \left( d^+_{G}(v) - \sum_{j=0}^{m} \tilde{E}(w_j, w_i) \right)\right) - E(\xi(w_{m+1}),v)\\
        &\le \left( \sum_{i = 0}^{m+1} \mathds{1}_{\xi(w_i) = v} \left( d^+_{G}(v) - \sum_{j=0}^{m} \tilde{E}(w_j, w_i) \right)\right) - \sum_{w' \in \hat{V}} \mathds{1}_{\xi(w') = v} \tilde{E}(w_{m+1}, w')\\
        &\le  \left( \sum_{i = 0}^{m+1} \mathds{1}_{\xi(w_i) = v} \left( d^+_{G}(v) - \sum_{j=0}^{m} \tilde{E}(w_j, w_i) \right)\right) - \sum_{i=0}^{m+1} \mathds{1}_{\xi(w_i) = v} \tilde{E}(w_{m+1}, w_i)\\
        &= \sum_{i = 0}^{m+1} \mathds{1}_{\xi(w_i) = v} \left( d^+_{G}(v) - \sum_{j=0}^{m+1} \tilde{E}(w_j, w_i) \right),
    \end{align*}
where we used the second condition at the fourth step. 

\medskip
\noindent
\textbf{Case 2:} $v \ne \xi(w_{m+1})$ and $\max(B_{\sigma}(m+1)(v)  , E(g_{\sigma}(m+1),v)) = E(g_{\sigma}(m+1),v)$.

\medskip
\noindent
By the third condition, we have  \[
    \mathds{1}_{\xi(w) = v} \left(d^+_{G}(v) - \sum_{j=0}^{m} \tilde{E}(w_j, w)\right) \ge 0
    \]
for all $w \in \hat{V}$ and $m \in [[0,P-2]] $. Therefore, one immediately gets 
\begin{align*}
B_{\sigma}(m+2)(v) &=  0 \\ &\le \sum_{i = 0}^{m+1} \mathds{1}_{\xi(w_i) = v} \left( d^+_{G}(v) - \sum_{j=0}^{m} \tilde{E}(w_j, w_i) \right)\\
& = \sum_{i = 0}^{m+1} \mathds{1}_{\xi(w_i) = v} \left( d^+_{G}(v) - \sum_{j=0}^{m+1} \tilde{E}(w_j, w_i) \right).
\end{align*}

\medskip
\noindent
\textbf{Case 3:} $v = \xi(w_{m+1})$.

\medskip
\noindent
Here we can use the fact that $\tilde{E}(w_j, w_i) = 0$ for all $j < i$ as well as $\tilde{E}(w_j, w_i) = 0$ if $\xi(w_i) = \xi(w_j)$ in order to get

\begin{align*}
        B_{\sigma}(m+2)(v) &= B_{\sigma}(m+1)(v) + d^+_G(v)\\
        &\le \left( \sum_{i = 0}^{m} \mathds{1}_{\xi(w_i) = v} \left( d^+_{G}(v) - \sum_{j=0}^{m} \tilde{E}(w_j, w_i) \right)\right) + d^+_G(v)\\
        &= \left( \sum_{i = 0}^{m} \mathds{1}_{\xi(w_i) = v} \left( d^+_{G}(v) - \sum_{j=0}^{m+1} \tilde{E}(w_j, w_i) \right)\right) + d^+_G(v) - \sum_{j=0}^{m+1} \tilde{E}(w_j, w_{m+1})\\
        &= \sum_{i = 0}^{m+1} \mathds{1}_{\xi(w_i) = v} \left( d^+_{G}(v) - \sum_{j=0}^{m+1} \tilde{E}(w_j, w_i) \right).
    \end{align*}
This completes the induction step. Now, by applying $(*)$ to $m = P-1$ and $v_k \in V$, we obtain
    \begin{align*}
        B_{\sigma}(P)(v_k) &\le \sum_{i = 0}^{P-1} \mathds{1}_{\xi(w_i) = v_k} \left( d^+_{G}(v_k) - \sum_{j=0}^{P-1} \tilde{E}(w_j, w_i) \right) = \sum_{i = 0}^{P-1} \mathds{1}_{\xi(w_i) = v_k} \left( d^+_{G}(v_k) - d^{-}_{\tilde{G}}(w_i) \right)\\
        &= \sum_{l=1}^{v_G(k)} \left(  d^+_{G}(v_k) - d^{-}_{\tilde{G}}(v_{k,l})\right) = d^+_G(v_k)*v_G(k) - \sum_{l=1}^{v_G(k)}d^{-}_{\tilde{G}}(v_{k,l}).
    \end{align*}
Therefore, by Theorem~\ref{Second formula} and Corollary~\ref{globalgrowth}, we deduce that 

\begin{align*}
        c& \le \sum_{v \in V} B_{\sigma}(P-1)(v)\\
        &\leq \sum_{v \in V} B_{\sigma}(P)(v) = \sum_{k = 1}^{N} B_{\sigma}(P)(v_k)\\
        &\le \sum_{k = 1}^{N} \left( d^+_G(v_k)*v_G(k) - \sum_{l=1}^{v_G(k)}d^{-}_{\tilde{G}}(v_{k,l}) \right)\\
        &= \left( \sum_{k = 1}^{N} d^+_G(v_k)*v_G(k) \right) - \sum_{v \in \hat{V}} d^{-}_{\tilde{G}}(v)\\
        &= \left( \sum_{k = 1}^{N} d^+_G(v_k)*v_G(k) \right) - a,
    \end{align*}
which concludes the first half of the proof. We shall now proceed with the second inequality, i.e. $\left( \sum_{k = 1}^{N} d^+_G(v_k)*v_G(k) \right) - a \leq c.$ Here we use this key observation: any primitive function is of the form $m \mapsto \xi(w_m)$ for a certain ordering $w_0,\dots,w_{P-1}$ of $\hat{V}$. Hence, by Theorem~\ref{Second formula}, we can choose an optimal ordering and the associated $\sigma$ so that $$c = \sum_{v \in V} B_{\sigma}(P-1)(v) = \sum_{v \in V} B_{\sigma}(P)(v),$$ the second equality being explained in Lemma~\ref{gaingood}. We define $\tilde{a}$ as the biggest number of edges of $\hat{G}$ that one can keep so that the associated subgraph $\tilde{G} = (\hat{V},\tilde{E}) $ respects the three following conditions:
    \begin{enumerate}
        \item[a)] For all $i < j$, $\tilde{E}(w_i, w_j) = 0 $
        \item[b)] For all $v \in V$ and $w \in \hat{V}$, $\sum_{w' \in \hat{V}} \mathds{1}_{\xi(w') = v} \tilde{E}(w,w') \le E(\xi(w), v)$
        \item[c)] For all $w \in \hat{V}$, $d^{-}_{\tilde{G}}(w) \le d^+_G(\xi(w)) $.
    \end{enumerate}
It is clear that a) implies acyclicity and therefore $\tilde{a} \leq a.$ In order to conclude, we will need to apply Lemma~\ref{stronglemma} to $\tilde{G}$ and, for that purpose, we have to show that $\tilde{G}$ satisfies the conditions $(1)$ to $(4).$ Obviously, we can pair $(1)$ with $a)$, $(2)$ with $b)$, $(3)$ with $c)$ and $(4)$ with the maximality of $\tilde{a}$. Hence, by applying Lemma~\ref{stronglemma} to $\tilde{G}$, we get \[
    B_{\sigma}(P)(v_i) \ge d^+_G(v_i)*v_G(i) - \sum_{l=1}^{v_G(i)}d^{-}_{\tilde{G}}(v_{i,l})
    \] for all $v_i \in V$ and, therefore,
    \begin{align*}
        c &= \sum_{i = 1}^{N} B_{\sigma}(P)(v_i)\\
        &\ge \sum_{i = 1}^{N} \left(d^+_G(v_i)*v_G(i) - \sum_{l=1}^{v_G(i)}d^{-}_{\tilde{G}}(v_{i,l}) \right)\\
        &= \left(\sum_{i=1}^N d^+_G(v_i) * v_G(i) \right) - \tilde{a}\\
        &\ge \left(\sum_{i=1}^N d^+_G(v_i) * v_G(i) \right) - a.
    \end{align*}
Hence the conclusion.
\end{proof}

\begin{remark}
It is legitimate to call the number $\left(\sum_{i=1}^N d^+_G(v_i) * v_G(i) \right) - a$ a \textit{primitive feedback number} since it gives back the feedback number if $G$ is Eulerian. Indeed, in that case, one has $v_G = (1,1,\dots,1)$ and, as a consequence, $\hat{G} = G$ and $\xi = \text{id}.$ The second condition of Theorem~\ref{Genfeednum} is thus trivially satisfied, while the third one is also satisfied since $d^-_{\tilde{G}}(w) \leq d^-_{G}(w) = d^+_G(w)$ by the Eulerian hypothesis. Hence, $c$ is simply equal to $\left(\sum_{i=1}^N d^+_G(v_i)\right)-a$ where $a$ is the biggest number of edges of $G$ that one can keep so that the associated subgraph is acyclic. This is exactly the feedback number of $G$.
\end{remark}

\end{document}